\documentclass[12pt]{amsart}
\usepackage{fancybox}
\usepackage{delarray}
\usepackage{enumerate}
\usepackage{amsmath,amssymb}
\usepackage{color}
\usepackage{comment}
\def\re{\mathbb{R}}

\def\N{\mathbb{N}}

\def\eps{\varepsilon}

\def\pd{\partial}

\def\la{\lambda}

\def\({\left(}
\def\){\right)}

\def\pd{\partial}

\def\BOX{{\setlength{\unitlength}{1pt}\begin{picture}(8,8)
	\put(1,1){\framebox(6,6)}\end{picture}}\ }
\def\qed{\hfill\BOX\vskip1em\par}

\def\intO{\int_{\Omega}}
\def\intdO{\int_{\partial\Omega}}

\def\al{\alpha}

\newtheorem{theorem}{Theorem}[section]
\newtheorem{corollary}[theorem]{Corollary}
\newtheorem{lemma}[theorem]{Lemma}
\newtheorem{proposition}[theorem]{Proposition}

\newtheorem{remark}[theorem]{Remark}

\begin{document}

\title[Sharp Logarithmic Sobolev inequality]{
Sharp Logarithmic Sobolev and related inequalities with monomial weights
}
\author{Filomena Feo}
\author{Futoshi Takahashi}

\date{\today}

\begin{abstract}
We derive a sharp Logarithmic Sobolev inequality with monomial weights starting from a sharp Sobolev inequality with monomial weights. 
Several related inequalities such as Shannon type and Heisenberg's uncertain type are also derived. 
A characterization of the equality case for the Logarithmic Sobolev inequality is given when the exponents of the monomial weights are all zero or integers. 
Such a proof is new even in the unweighted case.
\end{abstract}

\maketitle

\section{Introduction}

In the recent paper \cite{Cabre-RosOton}, the authors establish an isoperimetric inequality with monomial weights
and derive Sobolev, Morrey, and Trudinger inequalities from such geometric inequality.
More precisely, let $A = (A_1, A_2, \dots, A_n)$ be a nonnegative vector in $\re^n$, \textit{i.e.} $A_1 \ge 0, \dots, A_n \ge 0$, and define
\[
	\re^n_A = \{ x = (x_1, \dots, x_n) \in \re^n \, : \, x_i > 0 \ \text{if} \ A_i > 0 \}
\]
and the monomial weight
\[
	x^A = |x_1|^{A_1} \cdots |x_n|^{A_n} \text{ for } x \in \re^n_A.
\]
For any bounded open set $\Omega$ of $\re^n$ let us denote by  $W^{1,p}_0(\Omega, x^A dx)$ the closure of the space $C_c^1(\re^n)$ with the norm
$\| f \|_{W^{1,p}(\Omega, x^A dx)} = (\int_{\Omega} (|f|^p +|\nabla f|^p) x^A dx)^{1/p}$ for $1 \le p < \infty$. 
The Sobolev inequality with monomial weights proved in \cite{Cabre-RosOton} reads as follows:

\begin{theorem}{\rm (Sharp Sobolev inequality with monomial weights \cite{Cabre-RosOton})}
\label{Theorem:Sobolev}
Let $A = (A_1, A_2, \dots, A_n)$ be a nonnegative vector in $\re^n$, $D=n+A_1+\cdots+A_n, 1 \le p < D$ and $p_* = \frac{Dp}{D-p}$.
Then the inequality
\begin{equation}\label{Sobolev weight}
	\( \int_{\re^n_A} |f|^{p_{*}}\,x^A\, dx \)^{1/p_{*}} \le C_{p,n,A} \( \int_{\re^n_A} |\nabla f|^p \,x^A \,dx \)^{1/p}
\end{equation}
holds true for any $f \in W^{1,p}_0(\re^n, x^A dx)$, where
\begin{align}\label{C1}
	C_{1,n,A}= D^{-1} \( \frac{2^k\Gamma\(1 + \frac{D}{2}\)}{\Gamma\(\frac{A_1+1}{2}\)\Gamma\(\frac{A_2+1}{2}\)\cdots\Gamma\(\frac{A_n+1}{2}\)} \)^{\frac{1}{D}} \quad \text{for} \quad p = 1,
\end{align}
\begin{align} \label{Cp}
	C_{p,n,A}= C_{1,n,A} D^{1-\frac{1}{p}-\frac{1}{D}} \( \frac{p-1}{D-p} \)^{\frac{1}{p^{\prime}}} \( \frac{p^{\prime}\Gamma(D)}{\Gamma\(\frac{D}{p}\)\Gamma\(\frac{D}{p^{\prime}}\)} \)^{\frac{1}{D}} 
\quad \text{for} \quad 1 < p < D.
\end{align}
Here $k = \sharp \{ i \in \{ 1, \dots, n \} \, : \, A_i > 0 \}$ denotes the number of positive entries of the vector $A$, 
$\Gamma(s)$ denotes the Gamma function, and as usual $p^{\prime} = \frac{p}{p-1}$. 
Moreover, the constant $C_{1,n,A}$ is not attained by any function in $W^{1,1}_0(\re^n, x^A dx)$.
On the other hand,
the constant $C_{p,n,A}$ is attained in $W^{1,p}_0(\re^n, x^A dx)$ for $1 < p < D$ by functions of the form
\begin{equation} \label{estr sobolev}
	(a + b|x|^{\frac{p}{p-1}})^{1-D/p},
\end{equation}
where $a$ and $b$ are any positive constants.
\end{theorem}
When $A=(0,...,0)$, Theorem \ref{Theorem:Sobolev} reduces to the classical Sobolev inequality.
Unlike the classical one, the previous inequality is {\it not} invariant under the translation and the rotation of the space when $A \not\equiv (0, \cdots, 0)$, 
but is homogeneous and invariant with respect to the rescaling $f\rightarrow f_\lambda(x)=\lambda^{\frac{D-p}{p}}f(\lambda x)$ for $\lambda>0$. 
Note that it is established only when $A_i=0$ or $A_i\in\mathbb{N}$ for all $i$, that {\it all} extremizers which achieve the equality must be of the form \eqref{estr sobolev}. 
Moreover the best constant $C_{1,n,A}$ is the inverse of the corresponding best constant of the isoperimetric inequality with the monomial weight:
\begin{equation}\label{Dis Isop}
	\frac{P(\Omega)}{m(\Omega)^{\frac{D-1}{D}}} \ge \frac{P(B_1^A)}{m(B_1^A)^{\frac{D-1}{D}}} = C_{1,n,A}^{-1},
\end{equation}
where $\Omega \subset \re^n$ is a bounded Lipschitz domain,
\[
	m(\Omega) = \intO x^A dx, \quad P(\Omega) = \intdO x^A ds_x,
\]
and $B_1^A$ denotes the intersection of the unit ball $B_1(0) \subset \re^n$ with $\re^n_A$:
\begin{equation}\label{B1A}
B_1^A = \re^n_A \cap B_1(0).
\end{equation}

In this paper,
we derive a sharp Logarithmic Sobolev inequality with monomial weights starting from the sharp Sobolev inequality with monomial weights above. 
We follow the idea by Beckner and Pearson \cite{Beckner-Pearson}. 
As in \cite{Beckner-Pearson}, the product structure of both the Euclidean space and the weight (in our case), and the asymptotic behavior of the constant \eqref{Cp} as $p \to \infty$, are essential. 
Also several related inequalities such as Shannon type and Heisenberg's uncertain principle type are also derived.
A characterization of the equality case for the Logarithmic Sobolev inequality is given when the exponents of the monomial weights are all zero or integers. 
Such a proof is new even in the unweighted case.

First, we obtain the following theorem where $H^1_0(\re^n, x^A\,dx)$ denotes $W^{1,2}_0(\re^n, x^A\,dx)$.

\begin{theorem}{\rm (Sharp Logarithmic Sobolev inequality with monomial weights)}
\label{Theorem:log-Sobolev}
Let $A = (A_1, A_2, \dots, A_n)$ be a nonnegative vector in $\re^n$ and $D=n+A_1+\cdots+A_n$.  For any $f \in H^1_0(\re^n,x^A\,dx)$ such that $\int_{\re^n_A} |f|^2 x^A dx = 1$, the inequality
\begin{equation}
\label{log-Sobolev}
	\int_{\re^n_A} |f|^2 \log |f|^2 x^A dx \le \frac{D}{2} \log \( \frac{2}{\Pi(A)eD} \int_{\re^n_A} |\nabla f|^2 x^A dx \)
\end{equation}
holds true, where
\begin{equation}\label{Pi A}
	\Pi(A)= \left[ \frac{\prod_{i=1}^n \Gamma(\frac{A_i+1}{2})}{2^k} \right]^{2/D},
\end{equation}
and $k = \sharp \{ i \, : \, A_i > 0 \}$.
The equality in \eqref{log-Sobolev} holds if $f(x)=\frac{e^{-\frac{|x|^2}{4}}}{\left(2\Pi(A)\right)^{\frac{D}{4}}}$, which satisfies that $\int_{\re^n_A} |f|^2 x^A\,dx=1$ and $\int_{\re^n_A} |x|^2|f|^2 x^A\,dx=D$.
\end{theorem}

If we take $A = (0, \dots, 0)$, then  $k = 0$, $D = n$, and $\Pi(A) = \pi$, so we recover the classical Euclidean Logarithmic Sobolev inequality:
\begin{equation}
\label{log-Sobolev(Weissler)}
	\int_{\re^n} |f|^2 \log |f|^2 dx \le \frac{n}{2} \log \( \frac{2}{\pi e n} \int_{\re^n} |\nabla f|^2 dx \)
\end{equation}
for $f \in H^1_0(\re^n)$ with $\|f\|_{L^2(\re^n)} = 1$.
Stated in this form, \eqref{log-Sobolev(Weissler)} appears in a paper by Weissler \cite{Weissler}, 
but in terms of the Entropy power $N(g)=e^{-\frac{2}{n}\int_{\re^n} g \log g \, dx}$ and the Fisher information $I(g)=\int_{\re^n} \frac{|\nabla g|^2 }{g}dx$, the inequality 
$$
	\frac{1}{2\pi e}N(g)I(g)\geq n 
$$
goes back to Stam \cite{stam};
here $g$ is a positive function such that $\int_{\re^n} g\, dx = 1$. 
This is obtained by taking $f^2=g$ in \eqref{log-Sobolev(Weissler)}. 
By this inequality it follows that as information increases then the entropy (a measure of disorder) must increase also. 
For more information about the Logarithmic Sobolev inequalities we refer the reader to \cite{Gross}, \cite{delPino-Dolbeault}, \cite{Gentil03} and the book \cite{Lieb-Loss}.
Note that \eqref{log-Sobolev} is {\it not} invariant under the translation and the rotation of the space when $A \not\equiv (0, \cdots, 0)$, 
but is invariant with respect to the scaling $f\rightarrow f_\lambda(x)=\lambda^{\frac{D}{2}}f(\lambda x)$ for $\lambda>0$. 
Finally we stress that \eqref{log-Sobolev} cannot be obtained by a change of variables from the unweighted Logarithmic Sobolev inequality, even when $A_i\in\mathbb{N}$ for every $i=1,\cdots,n$.
This is different from the case for \eqref{Sobolev weight}.

\medskip
The characterization of the extremals for \eqref{log-Sobolev(Weissler)} is well-known (see \cite{Carlen}).
Here we propose a new (also in the unweighted case) and more elementary proof in the case when $A_i\in\mathbb{N}\cup\{0\}$ for every $i\in\{1,\cdots,n\}$.  

\begin{theorem}\label{Theorem:equality}
If $A_i\in\mathbb{N}\cup\{0\}$ for every $i\in\{1,\cdots,n\}$, then the equality in \eqref{log-Sobolev} occurs if and only if
\begin{equation}\label{Equality}
	f(x) = 
	\begin{cases}
	\(2\sigma\pi\)^{-\frac{n}{4}}e^{-\frac{|x-x_0|^2}{4\sigma}} & \mbox{ if } A_i=0 \quad \text{for any} \ i\in\{1,\cdots,n\}, \\
	\(2\sigma\Pi(A)\)^{-\frac{D}{4}}e^{-\frac{|x|^2}{4\sigma}} & \mbox{ if } A_i\in\mathbb{N} \quad \text{for some} \ i\in\{1,\cdots,n\}
	\end{cases}
\end{equation}
with $\sigma>0$ and $x_0\in \mathbb{R}^n$, respectively.
\end{theorem}

In order to explain the basic idea of the proof let us consider  $A = (0, \dots, 0)$. 
We take into account the following observations:

\begin{itemize}
\item[i)] Logarithmic Sobolev inequality can be obtained (see the proof of Theorem \ref{Theorem:log-Sobolev}) as a ``limit" of the Sobolev inequality for suitable functions;
\item[ii)] The equality case in the classical Sobolev inequality occurs if and only if the functions are of the form
\begin{equation}\label{classic equality}
	a^{-1}(1 + b|x-x_0|^{2})^{1-n/2},
\end{equation}
where $a\in\mathbb{R}-\{0\}$ and $b>0$ and $x_0\in\mathbb{R}^n$;
\item[iii)] The family of functions \eqref{classic equality} are densities of some generalized Cauchy distributions, 
which have the general form $a^{-1}(1 +b|x|^{2})^{-\beta}$ with $b,\beta>0$ and normalizing constant $a$ depending on $n$ an $\beta$. 
These probability measures may be considered (see \textit{e.g.} \cite{Bob}) as a natural \textquotedblleft pre-Gaussian model\textquotedblright, 
where the Gaussian case appears in the limit as $\beta\rightarrow + \infty$ (after proper rescaling of the coordinates).
\end{itemize}
\noindent Similar observations hold in the case $A \not\equiv (0, \cdots, 0)$.
However, since all extremals for \eqref{Sobolev weight} are given by \eqref{estr sobolev} only if $A_1,\cdots,A_n$ are integers or zero (see \cite{Cabre-RosOton}), 
we can derive a result only in this special case.

\medskip
As a corollary of Theorem \ref{Theorem:log-Sobolev}, we obtain a Nash type inequality with monomial weights as follows:
\begin{corollary} {\rm (Nash type inequality with monomial weights)}
\label{Corollary:Nash}

Let $A$ and $D$ be as in Theorem \ref{Theorem:log-Sobolev}. 
For any $f \in H^1_0(\re^n_A,  x^A \,dx)\cap L^1(\re^n_A,  x^A \,dx)$, we have the inequality
\begin{equation}
\label{Nash}
	\( \int_{\re^n_A} |f|^2 x^A dx \)^{1 + \frac{2}{D}} \le \ \frac{2}{\Pi(A)eD}  \( \int_{\re^n_A} |\nabla f|^2 x^A dx \) \( \int_{\re^n_A} |f| x^A dx \)^{\frac{4}{D}},
\end{equation}
where $\Pi(A)$ is defined in \eqref{Pi A}.
\end{corollary}

The unweighted version of \eqref{Nash} is one of the main tools used by J. Nash in \cite{Nash} on the H\"{o}lder regularity of solutions of divergence form uniformly elliptic equations.
It is well-known that the Nash inequality can also be derived by combining the H\"{o}lder and the Sobolev inequality. 
Indeed in our case we may use \eqref{Sobolev weight} and the following H\"{o}lder inequality
$$
	\|f\|_{L^2(\re^n_A,x^A\, dx)}\leq \|f\|^{\theta}_{L^{2^*}(\re^n_A,x^A\, dx)}\|f\|^{1-\theta}_{L^1(\re^n_A,x^A\, dx)}, 
$$
where $0<\theta<1$ and $1=\frac{2\theta}{2_*}+2(1-\theta)$ where $2_* = \frac{2D}{D-2}$.
Even in the unweighted case the constant in \eqref{Nash} is not sharp as observed in \cite{CL1993}.

\medskip
Finally we prove a ``dual" inequality of the Logarithmic Sobolev inequality with monomial weight \eqref{log-Sobolev}.

\begin{theorem}{\rm (Shannon type inequality with monomial weights)}
\label{Theorem:Shannon}
Let $A$ and $D$ be as in Theorem \ref{Theorem:log-Sobolev}. 
For any $f \in L^2(\re^n, x^A dx)$ with $\int_{\re^n_A} |f|^2 x^A dx = 1$ and $\int_{\re^n_A} |f|^2 |x|^2 x^A dx < \infty$,
the inequality
\begin{equation}
\label{L^2-Shannon}
	-\int_{\re^n_A} |f|^2 \log |f|^2 x^A dx \le \frac{D}{2} \log \( \frac{2 \Pi(A) e}{D} \int_{\re^n_A} |f|^2 |x|^2 x^A dx \)
\end{equation}
holds true.

More generally, for any $f \in L^1_{\al}(\re^n, x^A dx)$ with $\int_{\re^n_A} |f| x^A dx = 1$,
the inequality
\begin{equation}
\label{Shannon}
	-\int_{\re^n_A} |f| \log |f| x^A dx \le \frac{D}{\al} \log \( \frac{\al C_A(\al) e}{D} \int_{\re^n_A} |f| |x|^\al x^A dx \)
\end{equation}
holds true, where
\[
	L^1_{\al}(\re^n, x^A dx) = \{ f \in L^1(\re^n, x^A dx) \, | \, \int_{\re^n_A} |f| |x|^{\al} x^A dx < \infty \}
\]
for $\al > 0$ and
\begin{equation}
\label{C_A}
	C_A(\al) = \( \frac{\Gamma(\frac{D}{\al}+1)}{\Gamma(\frac{D}{2}+1)} \Pi(A)^{\frac{D}{2}} \)^{\al/D}.
\end{equation}
The equality in \eqref{Shannon} holds if $f(x) = \exp (-C_A(\al) |x|^{\al} )$ (up to scaling).
\end{theorem}

Note that in \eqref{L^2-Shannon} $C_A(2) = \Pi(A)$ and the equality holds for $f(x)=\frac{e^{-\frac{|x|^2}{4}}}{\left(2\Pi(A)\right)^{\frac{D}{4}}}$ 
which satisfies $\int_{\re^n_A} |f|^2 x^A\,dx=1$ and $\int_{\re^n_A} |x|^2|f|^2 x^A\,dx=D$. 
Moreover the inequality \eqref{Shannon} is invariant with respect to the scaling $f \mapsto f_{\la}(x) = \la^D f(\la x)$ for $\la > 0$.
An unweighted version of this Theorem appears in \cite{Ogawa-Seraku} and \cite{Ogawa-Wakui}.
Classical Shannon's inequality states that the normal distribution maximizes the Shannon Entropy among all distributions with fixed variance $\sigma^2$ and mean $\mu$. 
The inequality takes the form \eqref{Shannon} (without weight) since the Shannon  Entropy of normal distribution is $\frac{n}{2}\log\left(\frac{2\pi e}{n}\sigma^2\right)$.

\medskip
Inequalities \eqref{log-Sobolev} and \eqref{L^2-Shannon} give a lower and an upper bound of the entropy term. Indeed we have
\begin{align*}
	-\frac{D}{2} \log \( \frac{2 \Pi(A) e}{D} \int_{\re^n_A} |f|^2 |x|^2 x^A dx \)
	&\le \int_{\re^n_A} |f|^2 \log |f|^2 x^A dx \\
	&\le \frac{D}{2} \log \( \frac{2}{\Pi(A) e D} \int_{\re^n_A} |\nabla f|^2 x^A dx \)
\end{align*}
for $f \in H^1_0(\re^n, x^A dx)$ with $\int_{\re^n_A} |f|^2 x^A dx = 1$ and $\int_{\re^n_A} |f|^2 |x|^2 x^A dx < \infty$.
As an easy consequence we get the following corollary.
\begin{corollary}{\rm (Heisenberg's uncertainty principle type inequality with monomial weights)}
\label{Theorem:Heisenberg}
For any $f \in H^1_0(\re^n, x^A dx)$ with
\[
	\int_{\re^n_A} |f|^2 x^A dx = 1 \quad \text{and} \quad \int_{\re^n_A} |f|^2 |x|^2 x^A dx < \infty,
\]
the inequality
\[
	\frac{D^2}{4} \le \(\int_{\re^n_A} |f|^2 |x|^2 x^A dx \)\( \int_{\re^n_A} |\nabla f|^2 x^A dx \)
\]
holds true.
The equality holds for $f(x)=\frac{e^{-\frac{|x|^2}{4}}}{\left(2\Pi(A)\right)^{\frac{D}{4}}}$, 
which satisfies $\int_{\re^n_A} |f|^2 x^A\,dx=1$ and $\int_{\re^n_A} |x|^2|f|^2 x^A\,dx=D$.
\end{corollary}

The classical Heisenberg's uncertainty inequality, a precise mathematical formulation of the uncertainty principle of quantum mechanics, states that 
\[
	\( \int_{\re^n} |x|^2 |f(x)|^2 dx \) \( \int_{\re^n} |\nabla f(x)|^2 dx \) \ge \frac{n^2}{4} \( \int_{\re^n} |f(x)|^2 dx \)^2
\]
for any $f \in H^1(\re^n)$. 
Define the Fourier transform of $f$ as $\widehat{f}(\xi) = (2\pi)^{-n/2} \int_{\re^n} f(x) e^{-i x \cdot \xi} dx$ 
and recall that $\int_{\re}|\nabla f(x)|^2 dx = \int_{\re} |\xi|^2|\widehat{f}(\xi)|^2 \,d\xi$. 
Then it follows that $\(\int_{\re^n} |x-a|^2|f(x)|^2 \,dx \)\(\int_{\re^n} |\xi-b|^2|\widehat{f}(\xi)|^2 \,d\xi \) \ge \frac{n^2}{4}$ 
for any $f$ with $\int_{\re^n} |f(x)|^2 dx = 1$ and for any $a,b\in\re^n$. 
In other terminology, the above inequality reads ${\rm Var}(|f|^2) {\rm Var}(|\widehat{f}|^2) \ge n^2/4$, where
${\rm Var}(|f|^2) = \inf_{a \in \re^n} \int_{\re^n} |x-a|^2|f(x)|^2 \,dx$. 
The variance is a measure of the concentration of the probability density $|f|^2$. 
The more concentrated $f$ is around $a$, the smaller the variance will be. 
The above inequality states that if $f$ is concentrated around $a$, then $\widehat{f}$ cannot be concentrated around $b$, no matter which point $b$ in $\re^n$ we choose.
For more information on the uncertainty inequality, see \textit{e.g.} \cite{Folland-Sitaram}.
\medskip

The structure of the paper is as follows:
The proofs of all results stated here are given in $\S2$. 
Remarks and several related inequalities are discussed in $\S3$.

\section{Proofs}

First we collect here several lemmas which will be useful later.

Next lemma is an exercise of the book by W. Rudin \cite{Rudin} Chapter 3, page 71, 
see also \cite{Beckner3} page 122, and \cite{Park}.

\begin{lemma}\label{Lemma0}
Let $(X, \mu)$ be a measure space with $\mu(X) = 1$ and assume that $g \in L^r(X, \mu)$ for some $r > 0$.
Then it holds
\begin{equation}
\label{eq1}
	\lim_{p \to +0} \( \int_X |g|^p d\mu \)^{1/p} = \exp \( \int_X \log |g(x)| d\mu \)
\end{equation}
if $\exp(-\infty)$ is defined to be $0$.
\end{lemma}

\begin{lemma}
\label{Lemma:Gaussian}
For any $\al > 0$ and $t > 0$, we have
\begin{align}
\label{G1}
	&\int_{\re^n_A} e^{-t|x|^{\al}} x^A dx = t^{-\frac{D}{\al}} \frac{\Gamma(\frac{D}{\al}+1)}{\Gamma(\frac{D}{2}+1)} \Pi(A)^{\frac{D}{2}}, \\
\label{G2}
	&\int_{\re^n_A} e^{-t|x|^{\al}} |x|^{\al} x^A dx = \frac{D}{\al} t^{-\frac{D}{\al}-1} \frac{\Gamma(\frac{D}{\al}+1)}{\Gamma(\frac{D}{2}+1)} \Pi(A)^{\frac{D}{2}}.
\end{align}
In particular,
\[
	\int_{\re^n_A} e^{-C_A(\al) |x|^{\al}} x^A dx = 1, \quad \int_{\re^n_A} e^{-C_A(\al)|x|^{\al}} |x|^{\al} x^A dx = \frac{D}{\al C_A(\al)},
\]
where $C_A(\al)$ is defined by \eqref{C_A}.
\end{lemma}

\begin{proof}
Since \eqref{G2} is derived from \eqref{G1} by differentiating it with respect to $t$, we prove \eqref{G1} only.
Let $B_1^A$ be defined as in \eqref{B1A}
and put $x = (x_1, \cdots, x_n) = r \omega$, where $r = |x|$ and $\omega = (\omega_1, \cdots, \omega_n)$ be the unit vector in $\pd B_1^A = \pd B_1 \cap \re^n_A$.
Note that $\pd B_1^A$ is the curved part of the boundary portion of $B_1^A$.
Then $x^A = (r\omega)^A = (r\omega_1)^{A_1} \cdots (r\omega_n)^{A_n}$ and
\[
	x^A dx = (r\omega)^A r^{n-1} dr dS_\omega = r^{A_1 + \cdots + A_n + n-1} \omega^A \,dr \,dS_{\omega},
\]
where $dS_{\omega}$ denotes the surface measure on
 $\pd B_1^A$.
We calculate
\begin{align*}
	\int_{\re^n_A} e^{-|x|^{\al}} x^A dx &= \int_{\omega \in \pd B_1^A} \int_0^{\infty} e^{-r^\al} (r \omega)^A r^{n-1} dr dS_{\omega} \\
	&= \( \int_{\omega \in \pd B_1^A} \omega^A dS_{\omega} \) \int_0^{\infty} e^{-r^\al} r^{D-1} dr \\
	&= P(B_1^A) \frac{1}{\al} \Gamma\(\frac{D}{\al}\),
\end{align*}
where $P(B_1^A) = \int_{\omega \in \pd B_1^A} \omega^A dS_{\omega}$.
As observed in \cite{Cabre-RosOton} (Theorem 1.4 and Lemma 4.1), $P(B_1^A) = D m(B_1^A)$ and $m(B_1^A) = \int_{B_1^A} x^A dx = \frac{\Pi(A)^{\frac{D}{2}}}{\Gamma(\frac{D}{2}+1)}$.
Thus we obtain
\begin{align*}
	\int_{\re^n_A} e^{-|x|^{\al}} x^A dx = \frac{\Gamma\(\frac{D}{\al} + 1\)}{\Gamma\(\frac{D}{2}+1\)} \Pi(A)^{\frac{D}{2}}.
\end{align*}
The transformation $t^{1/\al} x = y$ for $x \in \re^n_A$ yields \eqref{G1}.
\end{proof}

\begin{lemma}\label{lemma computation}
Let $\sigma>0$. 
Then we have
\begin{equation}\label{int cauchy bis}
	\int_{\re^n_A} \frac{1}{(1 +\sigma |x|^2)^{\beta}}\,x^A \,dx = \(\frac{\Pi(A)}{\sigma}\)^{\frac{D}{2}}\frac{\Gamma\(\beta-\frac{D}{2}\)}{\Gamma(\beta)} \quad \hbox{ for } \beta>\frac{D}{2}.
\end{equation}
In particular,
\begin{equation}\label{int cauchy}
	\int_{\re^n} \frac{1}{(1 +\sigma|x-x_0|^2)^{\beta}} dx= \(\frac{\pi}{\sigma}\)^{\frac{n}{2}}\frac{\Gamma\(\beta-\frac{n}{2}\)}{\Gamma(\beta)} \quad \hbox{ for } \beta>\frac{n}{2}
\end{equation}
for any $x_0 \in \re^n$.
\end{lemma}

\begin{proof}
For \eqref{int cauchy} we also refer to \cite{Bob}. 
We derive \eqref{int cauchy bis} by using the ``polar coordinates" again.
As in the proof of the former lemma, we compute
\begin{align*}
	\int_{\re^n_A} \frac{x^A dx}{(1 +\sigma |x|^2)^{\beta}} &= \int_{\omega \in \pd B_1^A} \int_0^{\infty} \frac{(r \omega)^A r^{n-1} dr dS_{\omega}}{(1 + \sigma r^2)^{\beta}} \\
	&= \( \int_{\omega \in \pd B_1^A} \omega^A dS_{\omega} \) \int_0^{\infty} \frac{r^{D-1}}{(1 + \sigma r^2)^{\beta}} dr \\
	&= P(B_1^A) \(\frac{1}{\sigma} \)^{\frac{D}{2}} \int_0^{\infty} \frac{s^{D-1}}{(1 + s^2)^{\beta}} ds \\
	&= D \frac{\Pi(A)^{\frac{D}{2}}}{\Gamma(\frac{D}{2}+1)} \(\frac{1}{\sigma} \)^{\frac{D}{2}} \frac{\Gamma(\frac{D}{2}) \Gamma(\beta - \frac{D}{2})}{2\Gamma(\beta)} \\
	&= \(\frac{\Pi(A)}{\sigma} \)^{\frac{D}{2}} \frac{\Gamma(\beta - \frac{D}{2})}{\Gamma(\beta)}.
\end{align*}
\end{proof}

\subsection{Proof of Theorem \ref{Theorem:log-Sobolev}}

By density argument it is enough to prove our results for functions $f\in C_c^1(\re^n)$.

Let $N \in \N$ and let $B = (B_1, \dots, B_N) \in \re^N$ be a nonnegative vector.
Let us take $F \in C_c^1(\re^N)$ with $\int_{\re^N_B} |F(z)|^2 z^B dz = 1$ and let us denote $D_B = N + B_1 + \cdots + B_N$ and $2_*(B) = \frac{2D_B}{D_B-2}$.
Since the sharp $L^2$-Sobolev inequality with monomial weights \eqref{Sobolev weight} yields that for $D_B>2$
\begin{equation}
\label{sobolev F}
	\!\!\( \!\int_{\re^N_B} |F(z)|^{2_*(B)} z^B dz \!\)^{1/2_*(B)} \!\!\!\le\! C_{2,N,B} \(\! \int_{\re^N_B} |\nabla F(z)|^2 z^B dz \!\)^{1/2},
\end{equation}
where
\begin{equation}
\label{C2NB}
\!\!	C_{2,N,B} \!= \! D_B^{-\frac{1}{2}-\frac{1}{D_B}} \!\(\! \frac{\prod_{i=1}^N \Gamma(\frac{B_i+1}{2})}{2^{k_B}\Gamma(1 + \frac{D_B}{2})} \!\)^{-\frac{1}{D_B}}\!
\!\(\! \frac{1}{D_B-2} \!\)^{\frac{1}{2}} \!\(\! \frac{2\Gamma(D_B)}{\left[\Gamma(\frac{D_B}{2})\right]^2} \!\)^{\frac{1}{D_B}}\!\!
\end{equation}
and $k_B = \sharp \{ i \in \{ 1, \dots, N \} \, : \, B_i > 0 \}$, then we obtain
\begin{equation}
\label{L2}
	\!\!\log \!\( \!\int_{\re^N_B} |F(z)|^{2_*(B)-2} |F(z)|^{2}z^B dz \!\)^{\!\frac{1}{2_*(B)-2}}\!\!\! \le \! \log \!\left(\! C_{2,N,B}^2 \int_{\re^N_B} |\nabla F(z)|^2 z^B dz \!\right)^{\!\frac{D_B}{4}}\!\!\!.\!\!
\end{equation}

Let us consider a nonnegative vector $A = (A_1, \dots, A_n) \in \re^n$. We express $N = l n$ for $l \in \N$.
For a function $f \in C_c^1(\re^n)$ satisfying $\int_{\re^n_A} |f(x)|^2 x^A dx = 1$,
we put
\begin{equation}\label{F=}
	B = (\underbrace{A, A, \dots, A}_{l}) \in \re^{l n} = \re^N \quad \text{ and }\quad F(z) = \prod_{i=1}^l f(x^i),
\end{equation}
where $x^i = (x^i_1, \dots, x^i_n) \in \re^n$ for each $i = 1,2,\dots, l$,
and $z = (x^1, x^2, \cdots, x^l) \in \re^{ln} = \re^N$.
Note that for $z = (z_1, \dots, z_N) = (x^1, x^2, \dots, x^l) \in \re^N$ and $B = (A, A, \dots, A) \in \re^N$ as above,
we have the product structure of the space
\[
	\re^N_B = \re^n_A \times \dots \times \re^n_A
\]
and of the weight
$$
	z^B=(x^1)^A \times \dots \times (x^l)^A = \prod_{i=1}^l (x^i)^A.
$$
Moreover we stress that
$$D_B = N + B_1 + \dots + B_N = ln + l(A_1 + \dots + A_n) = l D,$$
where $D = n + A_1 + \dots + A_n$.
Under these notations, we have the next relations:
\begin{lemma}
\label{Prop:product}
Under the assumptions of this subsection we have
\begin{align*}
	&\int_{\re^N_B} |F(z)|^p z^B dz = \prod_{i=1}^l \int_{\re^n_A} |f(x^i)|^p (x^i)^A dx^i  \text{ for } p\geq1, \\
	&\int_{\re^N_B} |\nabla F(z)|^2 z^B dz = l \int_{\re^n_A} |\nabla f(x)|^2 x^A dx.
\end{align*}
\end{lemma}
The proof of this Lemma follows by a direct computation, so we omit it.
By this lemma, we have $\int_{\re^N_B} |F(z)|^2 z^B dz = 1$ for $F$ in \eqref{F=}.

\vspace{1em}
By Lemma \ref{Prop:product}, the inequality \eqref{L2} becomes
\begin{equation}\label{51}
l\log \( \int_{\re^n_A} |f(x)|^{2_*(B)-2} |f(x)|^{2}x^A dx \)^{\frac{1}{2_*(B)-2}} \le \frac{lD}{4} \log \left( l C_{2,N,B}^2 \int_{\re^n_A} |\nabla f(x)|^2 x^A dx \right)
\end{equation}
Now, by \eqref{C2NB} with $N = ln$ and $D_B = l D$ it follows that
$$
	l C_{2,N,B}^2=  \frac{1}{D}  \left[ \frac{2^k}{\prod_{i=1}^n \Gamma(\frac{A_i+1}{2})} \right]^{2/D} \frac{1}{lD-2} \left[ \frac{\Gamma(l D)}{\Gamma(\frac{l D}{2})} \right]^{\frac{2}{l D}}.
$$
Let $l \to \infty$ in the above equality. 
Stirling's formula
\begin{equation}\label{Stirling}
	\Gamma(s+1) = [1 + o(1)] \(s^s e^{-s} \sqrt{2\pi s} \) \quad \text{as} \quad s \to \infty,
\end{equation}
implies that
\begin{equation}\label{C l}
	\lim_{l \to \infty} l C_{2, ln, B}^2 = \frac{2}{eD} \left[ \frac{2^k}{\prod_{i=1}^n \Gamma(\frac{A_i+1}{2})} \right]^{2/D} = \frac{2}{\Pi(A)eD}.
\end{equation}
Now we apply Lemma \ref{Lemma0} with $g(x) = f(x)$, $d\mu = |f(x)|^2 x^A dx$, $X=\re^n_A$ and $p = 2_*(B) - 2 \to 0$ as $l \to \infty$.
Note that by the $L^2$-Sobolev inequality with monomial weight \eqref{Sobolev weight}, $\int_X |g|^{2_*(A)-2} d\mu < \infty$, thus we may take $r = 2_*(A) - 2$ in Lemma \ref{Lemma0}.   
Then Theorem \ref{Theorem:log-Sobolev} follows by taking a limit $l \to \infty$ in \eqref{51} with \eqref{C l}.

By Lemma \ref{Lemma:Gaussian} with $\alpha = 2$ and $t = 1/2$, we easily check that the equality in \eqref{log-Sobolev} holds for $f(x)=\frac{e^{-\frac{|x|^2}{4}}}{\left(2\Pi(A)\right)^{\frac{D}{4}}}$,
which satisfies that $\int_{\re^n_A} |f|^2 x^A\,dx=1$ and $\int_{\re^n_A} |x|^2|f|^2 x^A\,dx=D$.
\qed

\bigskip

\begin{remark}
We stress that it is the sharp asymptotics rather than the precise form of the Sobolev embedding constant that determines the Logarithmic Sobolev inequality. 
Indeed the constant in Theorem \ref{Theorem:Sobolev} is such that $C_{2,ln,B}\sim l^{-\frac{1}{2}}$ (up to a constant) as $l\rightarrow +\infty$, \textit{i.e.}, \eqref{C l}.
\end{remark}

\begin{remark}
The original proof by Beckner and Pearson \cite{Beckner-Pearson} used Jensen's inequality instead of Lemma \ref{Lemma0}. 
It works also in our framework.
Indeed Jensen's inequality implies that
\begin{align}
\label{L1}
	&\log \int_{\re^N_B} |F(z)|^{2_*(B)} z^B dz = \log \int_{\re^N_B} |F(z)|^{2_*(B)-2} |F(z)|^2 z^B dz \notag \\
	&\ge \frac{2_*(B)-2}{2} \int_{\re^N_B} |F(z)|^2 (\log |F(z)|^2 ) z^B dz.
\end{align}
For $F$ given by \eqref{F=}, we easily see that 
\begin{equation}\label{log l}
	\int_{\re^N_B} |F(z)|^2 (\log |F(z)|^2 ) z^B dz = l \int_{\re^n_A} |f(x)|^2 (\log |f(x)|^2 ) x^A dx.
\end{equation}
Recalling Lemma \ref{Prop:product} and combining \eqref{L1}, \eqref{log l}, and \eqref{L2}, we obtain
\begin{equation*}
	l \int_{\re^n_A} |f(x)|^2 (\log |f(x)|^2 ) x^A dx  \le \frac{l D}{2} \log \left\{ l C_{2,N,B}^2 \int_{\re^n_A} |\nabla f(x)|^2 x_A dx \right\}.
\end{equation*}
Theorem \ref{Theorem:log-Sobolev} follows form the inequality above with \eqref{C l}. 
\end{remark}

\begin{remark}
We stress that the isoperimetric inequality \eqref{Dis Isop}, or equivalently \eqref{Sobolev weight} with $p=1$ implies \eqref{log-Sobolev}. 
Indeed, let $U:\re^N\rightarrow\re$ be such that $\int_{\re^N_B} |U(z)| \,z^B dz =1$.
Jensen's inequality and  \eqref{Sobolev weight} with $p=1$ imply that
$$
	\int_{\re^N_B} |U(z)| \log |U(z)| \, z^B dz \le D_B \log\left[ C_{1,N,B}\int_{\re^N_B} |\nabla U(z)| z^B dz\right].
$$
Taking $U(z)= F^2(z) = \prod_{i=1}^Nf^2(x^i)$ where $F$ is in \eqref{F=} with $\|f\|_{L^2(\re^n_A, x^A\,dx)}=1$, 
and using H\"{o}lder's inequality, Lemma \ref{Prop:product} and \eqref{log l}, 
we have
$$
	l \int_{\re^n_A} |f(x)|^2 (\log |f(x)|^2 ) x^A dx  \le \frac{{l D}}{2}  \log \left[ 4 C^2_{1,N,B} l \int_{\re^n_A} |\nabla f(x)|^2 x_A dx \right].
$$
By Stirling formula \eqref{Stirling}, the inequality \eqref{log-Sobolev} follows.
\end{remark}

\subsection{Proof of Theorem \ref{Theorem:equality}}

We use the same notation introduced in the previous subsection.

It is easy to check that functions defined in \eqref{Equality} gives the equality in the Logarithmic Sobolev inequality. 
We want to prove that they are the only one. 
In order to do this we characterize the equality cases in every inequality in the proof of Theorem \ref{Theorem:log-Sobolev}. 
Without loss of generality we may consider only positive functions. 
For simplicity, first we consider the case $B=(0,..,0) \in \re^N$, \textit{i.e.} $D_B=N = nl$ and no weight is considered. 
Recall that extremals in the classical Sobolev inequality are all given by \eqref{classic equality}, see \cite{talenti best}. 
In order to have the equality for $F(z) = F_l(z)$ in \eqref{sobolev F},
it is necessary that, as $l \to \infty$,
$$
	F_l(z)\sim  a_l(1 + b_l|z-z_0^l|^{2})^{1-\frac{nl}{2}} \quad \hbox{ for every} z,\in\mathbb{R}^N
$$
with $a_l>0$, $b_l>0$, and $z_0^l \in \re^n$.
Here we have used the notation which emphasizes the dependence on $l$ of involved functions and constants.
Also we use the notation $\alpha_l \sim \beta_l$ if $\lim_{l \to \infty} \frac{\alpha_l}{\beta_l} = 1$.
By translation invariance, we may fix $z_0^l = z_0$ for a fixed point $z_0 \in \re^N$.
Also recalling that we consider function with $\int_{\re^N} |F_l(z)|^2 dz = 1$, by \eqref{int cauchy} we see that $a_l>0$ is related with $b_l$ as
\begin{align*}
	a_l^{-2}=\(b_l\)^{-\frac{ln}{2}}{\pi}^{\frac{nl}{2}}\frac{\Gamma\left(nl-2-\frac{nl}{2}\right)}{\Gamma(nl-2)}
\end{align*}
for $nl > 1$.
Stirling formula \eqref{Stirling} yields
\begin{align*}
	a_l^{\frac{1}{l}}\sim  \left(\frac{2b_lln}{\pi e}\right)^{\frac{n}{4}} \quad \hbox{ as } l\rightarrow+\infty.
\end{align*}
Choose  
$$
	z=(x,\underbrace{x_0,\cdots,x_0}_{l-1}), \quad z_0=(x_0,x_0,\cdots,x_0) \quad \text{with } x,x_0\in\re^n.
$$
Then it follows that
$$
	f_l(x)\sim e^{\frac{n}{4}\log\frac{2b_lln}{\pi e}-\frac{nl}{2}\log(1 +b_l|x-x_0|^{2})} \quad \hbox{ as } l\rightarrow+\infty,
$$
where $f_l$ is such that $F_l(z) = \prod_{i=1}^l f_l(x^i)$, $z = (x^1, \cdots, x^l) \in \re^{nl}$. 
Note that $(f_l(x_0))^l = F_l(z_0) \sim a_l$ as $l \to +\infty$.
We have three possible behaviors of the sequence $b_l$ as $l \to \infty$:
\begin{itemize}
\item [i)] $b_l\rightarrow+\infty,$
\item [ii)] $b_l\rightarrow \overline{b}\in \mathbb{R}-\{0\},$
\item [iii)]$b_l\rightarrow 0.$
\end{itemize}
Indeed if the limit does not exist, then we can argue one of the previous cases up to a subsequence.
The only non-trivial case to be considered is the third one, since the case i) or the case ii) occurs, 
then $f_l \sim 0$ as $l \to \infty$ which is absurd by the restriction $\int_{\re^n} |f_l(x)|^2 dx = 1$.
When iii) occurs, we have
$$
	f_l(x)\sim e^{\frac{n}{4}\log\frac{2b_lln}{\pi e}-\frac{nl}{2}b_l|x-x_0|^{2}}\quad \hbox{ as } l\rightarrow+\infty.
$$
Again three cases (up to a subsequence) are possible for the behaviors of the sequence $lb_l$,
but the only one to be considered is $lb_l\rightarrow \widetilde{b}\in \mathbb{R}-\{0\}$ (otherwise $f_l \sim 0$). 
Under this assumptions
$$
	f_l(x) \sim  \(\frac{2\widetilde{b}n}{e\pi}\)^{\frac{n}{4}} e^{-\frac{n\widetilde{b} }{2}|x-x_0|^{2}}\quad \hbox{ as } l\rightarrow+\infty
$$
up to a constant.
By $L^2$-normalization the characterization of the equality follows.

\medskip

A slight modification of this proof works well when monomial weights are taken into account.
In this case, since $z_0^l = 0$ for all $l \in \N$, the situation is simpler.
However as noticed before, we need the additional assumption that $A_i \in \N \cup \{ 0 \}$ for all $i \in \{ 1,\cdots, n \}$, 
in order to assure that all extremal functions of the Sobolev inequality with monomial weights \eqref{Sobolev weight} are of the form \eqref{estr sobolev}.
\qed

\subsection{Proof of Corollary \ref{Corollary:Nash}}
We follow the ``geometric" argument by Beckner \cite{Beckner2}.
We will derive the desired inequality \eqref{Nash} from Jensen's inequality and the Logarithmic Sobolev inequality with monomial weights.
As before, it is enough to prove Corollary for $f \in C_c^1(\re^n)$ by density.

Let $f \in C_c^1(\re^n)$ satisfying
\begin{itemize}
\item [i)] $f\not\equiv 0$ on any subset of $\re^n_A$ with positive measure,
\item [ii)] $\int_{\re^n_A} |f|^2 x^A dx = 1$.
\end{itemize}
By Jensen's inequality and the Logarithmic Sobolev inequality \eqref{log-Sobolev},
we have
\begin{align*}
	&-\log \( \int_{\re^n_A} |f| x^A dx \) = -\log \( \int_{\re^n_A} \frac{1}{|f|} |f|^2 x^A dx \) \\
	&\le \int_{\re^n_A} |f|^2 \log |f| x^A dx \le \frac{D}{4} \log \( \frac{2}{\Pi(A)eD} \int_{\re^n_A} |\nabla f|^2 x^A dx \).
\end{align*}
Thus by the monotonicity of $\log$-function, we get
\[
	1 \le\frac{2}{\Pi(A)eD} \( \int_{\re^n_A} |\nabla f|^2 x^A dx \) \( \int_{\re^n_A} |f| x^A dx \)^{\frac{4}{D}}.
\]
By homogeneity we get \eqref{Nash}. 
To avoid the assumption $i)$ it is enough to integrate on $\re^n_A\backslash\{f=0\}$ and observe that $\int_{\re^n_A\backslash\{f=0\}} |f|^2 x^A dx=\int_{\re^n_A} |f|^2 x^A dx=1$.
\qed

\begin{remark}
As observed in $\S3$ of \cite{Beckner3} the homogeneity and the dilation invariance of \eqref{Nash} allow us to use the convexity of the function $G_f(p)$ 
defined by $G_f(p) = \log \int_{\re^n_A} |f(x)|^p x^A \, dx$.
Indeed, for $f$ with $\|f\|_{L^2(\re^n_A, x^A dx)}=1$, define $f_{\lambda}(x) = \la^{D/2} f(\la x)$ with $\la^{D/2} =  \|f\|_{L^1(\re^n_A, x^A dx)}$.
Then we see $\|f_{\la}\|_{L^2(\re^n_A, x^A dx)}= \|f_{\la}\|_{L^1(\re^n_A, x^A dx)}=1$, and we have 
$$
	G_{f_{\la}}'(p) \Big |_{p=2} = \int_{\re^n_A} |f_{\la}|^2 (\log |f_{\la}|) x^A dx\geq0. 
$$ 
Then \eqref{Nash} follows from \eqref{log-Sobolev} for $f_{\la}$.
\end{remark}

\subsection{Proof of Theorem \ref{Theorem:Shannon}}

We give a proof of Theorem \ref{Theorem:Shannon} along the line of \cite{Ogawa-Seraku}.
It is enough to prove \eqref{Shannon} for any $f \in L^1_{\al}(\re^n, x^A dx)$ with $\int_{\re^n_A} |f| x^A dx = 1$,
since \eqref{L^2-Shannon} is derived by putting $\al = 2$ and $|f|^2$ in \eqref{Shannon} instead of $|f|$.

For $\al > 0$, put $\phi_{\al}(x) = e^{-C |x|^\al}$, where $C=C_A(\al),$ defined as in \eqref{C_A},
is chosen so that $\int_{\re^n_A} \phi_{\al}(x) x^A dx = 1$ (see Lemma \ref{Lemma:Gaussian}).
For $f \in L^1_{\al}(\re^n, x^A dx)$ with $\int_{\re^n_A} |f| x^A dx = 1$,
Jensen' s inequality implies that
\begin{align*}
\exp\left( \int_{\re^n_A} |f| \log (\phi_{\al}) x^A dx\right. &- \left. \int_{\re^n_A} |f| \log |f| x^A dx \right)
	\\
&= \exp \( \int_{\re^n_A} |f| \log \(\frac{\phi_{\al}}{|f|}\) x^A dx \) \\
	&\le \int_{\re^n_A} \exp \(\log \(\frac{\phi_{\al}}{|f|}\) \) |f| x^A dx
\\
	&= \int_{\re^n_A} \phi_{\al}\, x^A \,dx = 1.
\end{align*}
From this, we have
\begin{align*}
	\int_{\re^n_A} |f| \log \phi_{\al}\, x^A \,dx \le \int_{\re^n_A} |f| \log |f|\, x^A \,dx.
\end{align*}
Thus we obtain
\begin{equation}\label{S1}
	-\int_{\re^n_A} |f| \log |f|\, x^A\, dx \le -\int_{\re^n_A} |f| \log \phi_{\al}\, x^A\, dx
	= C_A(\al) \int_{\re^n_A} |f|\, |x|^{\al}\, x^A \,dx.
\end{equation}

Now, put $f_{\la}(x) = \la^D f(\la x)$ for $\la > 0$.
It is easy to check that
\[
	\int_{\re^n_A} |f_{\la}(x)| \,x^A \,dx = \int_{\re^n_A} |f(y)|\, y^A \,dy = 1.
\]
We insert $f_{\la}$ instead of $f$ into \eqref{S1}.
Since
\[
	-\int_{\re^n_A} |f_{\la}| \log |f_{\la}|\, x^A\, dx =  (-D \log \la) \int_{\re^n_A} |f(y)| \,y^A \,dy - \int_{\re^n_A} |f| \log |f|\, y^A \,dy
\]
and
\[
	\int_{\re^n_A} |f_{\la}(x)| |x|^{\al}\, x^A \,dx = \la^{-\al} \int_{\re^n_A} |f(y)| \,|y|^{\al} \,y^A \,dy,
\]
we have
\[
	(-D \log \la) \int_{\re^n_A} |f| \, y^A \,dy - \int_{\re^n_A} |f| \log |f|\, y^A\, dy \le C_A(\al) \la^{-\al} \int_{\re^n_A} |f(y)| \,|y|^{\al} \,y^A\, dy.
\]
Since we assume $\int_{\re^n_A} |f(y)| y^A dy = 1$, we have
\begin{equation}
\label{S2}
	- \int_{\re^n_A} |f| \log |f| \, y^A \, dy \le C_A(\al) \la^{-\al} \int_{\re^n_A} |f(y)|\, |y|^{\al} \,y^A \, dy + D \log \la := G(\la).
\end{equation}
We optimize $G(\la)$ with respect to $\la > 0$.
Denoting $J(f) = \int_{\re^n_A} |f(y)| \,|y|^{\al} \,y^A \,dy$, then $G(\la) = C_A(\al) J(f) \la^{-\al} + D \log \la$ for $\la > 0$.
An easy computation shows that $G(\la)$ has the unique minimum point when $\la = \la_* = \( \frac{\al C_A(\al) J(f)}{D} \)^{1/\al}$,
and the global minimum value is
\[
	G(\la_*) = \frac{D}{\al} \log \( \frac{C_A(\al) \al e}{D} J(f) \).
\]
Returning to \eqref{S2}, we obtain the inequality
\[
	- \int_{\re^n_A} |f| \log |f|\, y^A \,dy \le G(\la_*) = \frac{D}{\al} \log \( \frac{C_A(\al) \al e}{D} \int_{\re^n_A} |f(y)| \,|y|^{\al} \,y^A \,dy \).
\]

Concerning the equality case, Lemma \ref{Lemma:Gaussian} implies that
\begin{align*}
	&\int_{\re^n_A} \phi_{\al}(x) \log (\phi_{\al}^{-1}(x)) x^A dx 
= \frac{D}{\al}
=\frac{D}{\al} \log \( \frac{C_A(\al) \al e}{D} \int_{\re^n_A} \phi_{\al}(x)  |x|^{\al} x^A dx \)
.
\end{align*}
Thus $\phi_{\al}$ realizes the equality in \eqref{Shannon}.
This completes the proof.
\qed

\section{Some remarks}

In this section, we discuss about several inequalities related to our former results. 
For other inequalities such as Hardy-Sobolev type or Trudinger-Moser type with monomial weights, see \cite{Castro} and \cite{Lam}.

\subsection{Inequalities on the whole space}

It is easy to check that Theorem \ref{Theorem:Sobolev} holds on the whole $\re^n$ without the best constant.

\begin{corollary}
\label{Theorem:Sobolev W1p}
Let $A,D$ and $p_*$ be as in Theorem \ref{Theorem:Sobolev}.Then
\[
	\( \int_{\re^n} x^A |u|^{p_{*}} dx \)^{1/p_{*}} \le K \( \int_{\re^n} x^A |\nabla u|^p dx \)^{1/p}
\]
holds true for any $u \in W^{1,p}_0(\re^n, x^A dx)$, where
$$
	K=\left\{
	\begin{array}
	[c]{ll}%
	2^{\frac{1}{D}}C_{p,n,A}  &\quad \mbox{ if } \quad 1 \le p < D/2,
	\\
	C_p & \quad
	\mbox{ if }  \quad D/2 \le p < D.
	\end{array}
	\right.  
$$
\end{corollary}

\begin{proof}
Let $k = \sharp \{ i \in \{ 1, \dots, n \} \, : \, A_i > 0 \}$.
We apply \eqref{Sobolev weight} for $2^k$ hyperoctants $Q_i$ of $\re^n$, each $Q_i$ is a copy of $\re^n_A$ and $\cup_{i=1}^{2^k} Q_i = \re^n \setminus \cup_{i=1}^n \{\text{$x_i$-axis}\}$:
 \[
	 \int_{\re^n} x^A |u|^{p_{*}} dx =  \sum_{i=1}^{2^k}\int_{Q_i} x^A |u|^{p_{*}} dx \leq C^{p_{*}}_{p,n,A} \sum_{i=1}^{2^k} \( \int_{Q_i} x^A |\nabla u|^p dx \)^{p_{*}/p}.
\]
The assertion follows by observing that
$$
	a^q+b^q\leq (a+b)^q\, \hbox{ if } q\geq2 \, \hbox{ for  } a,b>0
$$
and
$$
	a^q+b^q\leq 2^{q-1}(a+b)^q\, \hbox{ if } 1<q<2 \, \hbox{ for  } a,b>0
$$
with $q = \frac{p_*}{p} = \frac{D}{D-p}$.
\end{proof}

We can also derive a Logarithmic Sobolev inequality \eqref{log-Sobolev} (not in sharp form) and a Nash type inequality 
when the domain of integration is whole $\re^n$, starting from Corollary \ref{Theorem:Sobolev W1p} with $p=2$ instead of Theorem \ref{Theorem:Sobolev}.

\subsection{$L^p$-Logarithmic Sobolev inequalities}

Using the $L^p$-Sobolev inequality it is possible to derive (in general not sharp) $L^p$-version of Logarithmic Sobolev inequality.
For $p=1$ we obtain a sharp inequality.

\begin{proposition}
\label{Theorem:log-Sobolev 1}
Let $A$ and $D$ be as in Theorem \ref{Theorem:log-Sobolev}.
The following inequality
\begin{equation}
\label{log-Sobolev 1}
	\int_{\re^n_A} |f| (\log |f|) x^A dx \le D \log \( K \int_{\re^n_A} |\nabla f| x^A dx \)
\end{equation}
holds true for any $f \in W^{1,1}_0(\re^n, x^A dx)$ such that $\int_{\re^n_A} |f| x^A dx = 1$, 
with the sharp constant given by $K= C_{1,n,A}$ defined as in \eqref{C1}.
Also no function in $W^{1,1}_0(\re^n, x^A dx)$ achieves the equality in \eqref{log-Sobolev 1}.
\end{proposition}

When $D>1$, \eqref{log-Sobolev 1} is obtained by combining Jensen's inequality and \eqref{Sobolev weight} with $p=1$. 
Then as an easy consequence, we see $C_{1,n,A}$ is an upper bound for $K$. 
Let $\chi_{B^A_1}(x)$ be the characteristic function of $B_1^A$ defined as in \eqref{B1A}. 
Since  $C_{1,nA}^{-1}$ is the best constant of the isoperimetric inequality \eqref{Dis Isop}, 
$f(x)=\frac{1}{\int_{B_1^A}  x^A dx }\chi_{B^A_1}(x)$ gives the equality in \eqref{log-Sobolev 1}. 
It follows that the constant $C_{1,n,A}$ is sharp in \eqref{log-Sobolev 1}.
When $D=1$ we reduce to the case $n=1$ without weight and Proposition follows by Theorem 2 of \cite{Beckner3}, 
where the unweighted version of this result is proved.

\medskip

The above argument does not give the sharp result for $p\neq1$.

\begin{proposition}{\rm ($L^p$ Logarithmic Sobolev inequality with monomial weights)}
\label{Theorem:log-Sobolev p}
Let $A$ and $D$ as in Theorem \ref{Theorem:log-Sobolev} and $1< p<D$.
Then the following inequality
\[
	\int_{\re^n_A} |f|^p (\log |f|^p) x^A dx \le \frac{D}{p} \log \( C^p_{p,n,A} \int_{\re^n_A} |\nabla f|^p x^A dx \)
\]
holds true for any $f \in W^{1,p}_0(\re^n)$ such that $\int_{\re^n_A} |f|^p x^A dx = 1$, where
$C_{p,n,A}$ is defined as in \eqref{Cp}.
\end{proposition}

The proof of Proposition  \ref{Theorem:log-Sobolev p} runs again by combining Jensen's inequality and \eqref{Sobolev weight}.
Even in the unweighted case the constant is not sharp (see \cite{delPino-Dolbeault}). 
When $A\equiv(0,\cdots,0) \in \re^n$, sharp $L^p$ versions of \eqref{log-Sobolev(Weissler)} are proved in \cite{delPino-Dolbeault} and in \cite{Gentil03}.

As shown in $\S2$ of \cite{Beckner3}, 
even in the unweighted case the asymptotic behavior of the constant, the product structure of $\re^n$, and of the weight for $p\neq 2$ don't allow us to obtain sharp inequalities, 
but only
$$
	\int_{\re^n} |g|^p (\log |g|^p) x^A dx \le \frac{n}{2} \log \( \frac{p^2}{2\pi e n} \int_{\re^n} |\nabla g|^2 |g|^{p-2} x^A dx \)
$$
for $g$ such that $\|g\|_{L^p(\re^n)}=1$. 
We point out that the last inequality can be derived from \eqref{log-Sobolev} by setting $|f|=|g|^{p/2}$.

\subsection{Logarithmic Sobolev trace inequality with monomial weights}

In this subsection we obtain a Logarithmic Sobolev trace inequality with monomial weights (see \cite{FP2013}, \cite{Park} for similar inequalities without weights).

\begin{proposition}{\rm (Logarithmic Sobolev trace inequality with monomial weights)}
\label{Pr:TraceSobolev}
Let $A$ be a nonnegative vector in $\re^n$, $D = n + A_1 + \dots + A_n$, $1\leq p<D+1$.
Then the inequality
\begin{equation*}
\begin{split}
	\int_{\re^n_A} & |f(x,0)|^p \left(\log |f(x,0)|^p\right)x^A\,dx  \le \\
	&\frac{D}{p}\log \( C_{p,n+1,A'}^{\frac{(D+1-p)(p-1)}{Dp}}  \int_{\re^n_A\times (0,+\infty)} |\nabla_{x,y} f(x,y)|^p x^A \,dx\,dy \)^{1/p}
\end{split}
\end{equation*}
holds true for any $f \in C_c^{1}(\re^{n+1})$ with $\|f(\cdot ,0)\|_{L^2(\re^n_A,x^A\, dx)}=1$, 
where $C_{p,n+1,A'}$ is defined in \eqref{Cp} (replacing $D$ by $D+1$, and $n$ by $n+1$) and $A'=(A_1,\cdots,A_n,0)$.
\end{proposition}

The proof of Proposition \ref{Pr:TraceSobolev} comes from Jensen's inequality and the following result.

\begin{lemma}{\rm (Sobolev trace inequality with monomial weights)}
\label{Pr:Trace}
Let $A, A', D$ and $C_{p,n+1,A'}$ be as in Proposition \ref{Pr:TraceSobolev} and $q=\frac{Dp}{D+1-p}$.
Then the inequality
\begin{equation}\label{Tr Sob}
\begin{split}
	\( \int_{\re^n_A} |f(x,0)|^{q} x^A \,dx \)^{1/q}&
\le q^{1/q}C_{p,n+1,A'}^{\frac{(D+1-p)(p-1)}{Dp}}
\\
&\( \int_{\re^n_A\times (0,+\infty)} |\nabla_{x,y} f(x,y)|^p x^A \,dx\,dy \)^{1/p}
\end{split}
\end{equation}
holds true for any $f \in C_c^{1}(\re^{n+1})$.
\end{lemma}

By a standard scaling argument one sees that the exponent $q$ is optimal, in the sense that this inequality cannot hold with any exponent $\tilde{q}$ different from $q$.

\begin{proof}[Proof of lemma \ref{Pr:Trace}]
We follow the idea of \cite{Kufner book}. 
We observe that 
$$
	|f(x,0)|^q\leq q \int_0^{+\infty}|f(x,\xi)|^{q-1}\left|\frac{\partial}{\partial \xi}f(x,\xi)\right|d\xi.
$$
The H\"{o}lder inequality yields
\begin{equation*}
\begin{split}
&\left( \int_{\re^n_A} x^A |f(x,0)|^{q} dx \right)^{1/q}\leq q^{1/q}
\left( \int_{\re^n_A} x^A \int_0^{\infty}|f(x,\xi)|^{q-1}\left|\frac{\partial}{\partial \xi}f(x,\xi)\right| d\xi dx \right)^{1/q}
\\
&\leq q^{1/q}
\left( \int_{\re^n_A} x^A \int_0^{\infty}|f(x,\xi)|^{\frac{(D+1)p}{D+1-p}}d\xi dx \right)^{\frac{(D+1-p)(p-1)}{Dp^2}}
\left( \int_{\re^n_A} x^A \int_0^{\infty}\left|\frac{\partial}{\partial \xi}f(x,\xi)\right|^p d\xi dx \right)^{\frac{1}{p}\frac{D+1-p}{pD}}
\\
&:=q^{1/q}I_1 I_2.
\end{split}
\end{equation*}
It is easy to get
$$
	I_2 \leq \( \int_{\re^n_A \times (0, +\infty)} |\nabla_{x,y}f|^p x^A dxdy \)^{\frac{D+1-p}{p^2 D}}.
$$
Let $\widetilde{f}$ be the even extension of $f$ to $\re^{n+1}$. 
Note that if we put $\tilde{x} = (x, \xi) \in \re^{n+1}$, we have $\tilde{x}^{A'} = x^A$.
Thus by the Sobolev inequality \eqref{Sobolev weight} applied for $W^{1,p}_0(\re^n_A \times \re, \tilde{x}^{A'} dx d\xi)$, 
we have
\begin{equation*}
\begin{split}
&I_1
=
\left( \frac{1}{2}\int_{\re^{n}_A\times\re} x^A |\widetilde{f}(x,\xi)|^{\frac{(D+1)p}{D+1-p}}d\xi dx \right)^{\frac{(D+1-p)(p-1)}{Dp^2}}
\\
&\leq
C_{p,n+1,A'}^{\frac{(D+1-p)(p-1)}{Dp}}
\left( \frac{1}{2}\int_{\re^{n}_A\times\re} x^A |\nabla_{x,y}\widetilde{f}(x,\xi)|^{p}d\xi dx \right)^{\frac{1}{p}\frac{(D+1)(p-1)}{Dp}}
\\
&=C_{p,n+1,A'}^{\frac{(D+1-p)(p-1)}{Dp}}
\left( \int_{\re^{n}_A} x^A |\nabla_{x,y}f(x,\xi)|^{p}d\xi dx \right)^{\frac{1}{p}\frac{(D+1)(p-1)}{Dp}}.
\end{split}
\end{equation*}
Combining this with the previous estimate we obtain the assertion.
\end{proof}
The above proof does not give a sharp constant in \eqref{Tr Sob}, so the obtained Logarithmic Sobolev trace inequality with monomial weights is also not sharp.

\subsection{From Trudinger-Moser inequality to Logarithmic Sobolev type inequality}

In this section, we derive a Logarithmic Sobolev-type inequalities from the sharp Trudinger-Moser inequality with monomial weights obtained recently by Lam \cite{Lam}.

\begin{proposition}(Lam \cite{Lam})
\label{Prop:Lam}
Let $\Omega \subset \re^N$ be a bounded domain. 
Then there exists a constant $C_0=C_0(D) > 0$ such that
\begin{equation}
\label{Lam_TM}
	\int_{\Omega} \exp \( \al |u|^{\frac{D}{D-1}} \) x^A dx \le C_0 m(\Omega)
\end{equation}
holds true for any $u \in W^{1,D}_0(\Omega, x^A dx)$ and $\al \le \al_D(A)$, where
$m(\Omega) = \int_{\Omega} x^A dx$ and
\begin{equation}\label{alpaD}
	\al_D(A) = D P(B_A)^{\frac{1}{D-1}} = D \left[ D \frac{\Pi_{i=1}^N \Gamma(\frac{A_i+1}{2})}{2^k\Gamma(1 + \frac{D}{2})} \right]^{\frac{1}{D-1}}.
\end{equation}
\end{proposition}

A similar result is proved in \cite{Cabre-RosOton} only for sufficiently small $\al$.

Using Proposition \ref{Prop:Lam}, first we obtain an improvement of the Sobolev embedding of
$W^{1,D}_0(\Omega, x^A dx)$ into $L^q(\Omega, x^A dx)$ for any $2 \le q < \infty$.

\begin{proposition}
\label{Prop:Refined_Sobolev}
For any $q \ge 2$, there exists $C(q) > 0$ such that
\begin{equation}
\label{Eq:Refined_Sobolev}
	\| u \|_{L^q(\Omega, x^A dx)} \le C(q) q^{\frac{D-1}{D}} \| \nabla u \|_{L^D(\Omega, x^A dx)}
\end{equation}
holds true for any $u \in W^{1,D}_0(\Omega, x^A dx)$.
Moreover, we have
\begin{equation}\label{lim C}
	\lim_{q \to \infty} C(q) = \left[\frac{D-1}{D\al_D(A)e}\right]^{\frac{D-1}{D}},
\end{equation}
where $\al$ is defined in \eqref{alpaD}.
\end{proposition}

\begin{proof}
We argue as in \cite{Ren-Wei(JDE)} Lemma 2.1.
Let $u \in W^{1,D}_0(\Omega, x^A dx)$.
We recall the following elementary inequality
\begin{equation}\label{88}
\frac{x^s}{\Gamma(s+1)} \le e^x \quad \forall x \ge 0, \forall s \ge 0.
\end{equation}
By \eqref{88} and \eqref{Lam_TM} we obtain
\begin{align*}
	&\frac{1}{\Gamma((\frac{D-1}{D})q+1)} \intO |u|^q x^A dx \\
	&= \frac{1}{\Gamma((\frac{D-1}{D})q+1)}
	\intO \left[ \al_D(A) \( \frac{|u|}{\| \nabla u \|_{L^D(\Omega, x^A dx)}} \)^{\frac{D}{D-1}} \right]^{(\frac{D-1}{D})q} x^A dx \\
	&\times \al_D(A)^{-\frac{D-1}{D}q} \| \nabla u \|_{L^D(\Omega, x^A dx)}^q  \\
	&\!\le \al_D(A)^{-\frac{D-1}{D}q} \| \nabla u \|_{L^D(\Omega, x^A dx)}^q\!\!\intO \exp \left[ \al_D(A) \( \frac{|u|}{\| \nabla u \|_{L^D(\Omega, x^A dx)}} \)^{\frac{D}{D-1}} \right] x^A dx
\\
	&\le C_0 m(\Omega) \al_D(A)^{-\frac{D-1}{D}q} \| \nabla u \|_{L^D(\Omega, x^A dx)}^q.
\end{align*}
Set
\[
	C(q) = \left[ \Gamma(\frac{D-1}{D}q+1) \right]^{1/q} C_0^{1/q}m_A(\Omega)^{1/q} \al_D(A)^{-(\frac{D-1}{D})} q^{-(\frac{D-1}{D})}.
\]
Stirling formula \eqref{Stirling} implies
\[
	\Gamma\( \frac{D-1}{D}q+1 \)^{1/q} \sim \(\frac{D-1}{De}\)^{\frac{D-1}{D}} q^{\frac{D-1}{D}}.
\]
Thus we have \eqref{Eq:Refined_Sobolev} and \eqref{lim C} holds.
\end{proof}

Now, we derive a Logarithmic Sobolev-type inequality from Proposition \ref{Prop:Refined_Sobolev}.

\begin{proposition}
Let $\Omega \subset \re^N$ be a bounded domain and let $q > D$.
Then
\begin{equation}\label{log Sob Omega}
	\int_{\Omega} |u|^D (\log |u|^D) x^A dx \le \frac{Dq}{q-D} \log \( C(q) q^{\frac{D-1}{D}} \( \int_{\Omega} |\nabla u|^D x^A dx \)^{1/D} \)
\end{equation}
holds true for any $u \in W^{1,D}_0(\Omega, x^A dx)$ with $\intO |u|^D x^A dx = 1$.
\end{proposition}

\begin{proof}
By Jensen's inequality and \eqref{Eq:Refined_Sobolev}, we get
\begin{align*}
	\intO \log (|u|^{q-D} ) |u|^D x^A dx &\le \log \( \intO |u|^{q-D} |u|^D x^A dx \) = \log \( \intO |u|^q x^A dx \) \\
	&\le \log \( C(q)^q q^{\frac{D-1}{D}q} \| \nabla u \|_{L^D(\Omega, x^A dx)}^q \),
\end{align*}
namely \eqref{log Sob Omega}.%
\end{proof}

{\bf Acknowledgment}:
This work has been partially supported by JSPS Grant-in-Aid for Scientific Research (B), No.19136384 (T.F), 
by GNAMPA - INdAM (F.F).


\vspace{1em}\noindent
{\it
Department of Mathematics, Graduate School of Science \& OCAMI, Osaka City University, 3-3-138 Sugimoto, Sumiyoshi-ku, Osaka, 558-8585, Japan}

\vspace{1em}\noindent
e-mail:futoshi@sci.osaka-cu.ac.jp

\vspace{1em}\noindent
{\it
Dipartimento di Ingegneria, Universit\`{a} degli Studi di
Napoli \textquotedblleft Parthenope\textquotedblright, Centro
Direzionale Isola C4 80143 Napoli, Italy}

\vspace{1em}\noindent
 e-mail:
filomena.feo@uniparthenope.it

\end{document}